\documentclass[11pt]{article}
\usepackage[margin=1.5in]{geometry} 
\usepackage{amssymb,amsfonts,amsmath,bbm,mathrsfs,stmaryrd}
\usepackage{xcolor}
\usepackage{url}
\usepackage{verbatim}

\usepackage[colorlinks,
             linkcolor=black!40!red,
             citecolor=blue,
             pdftitle={Top_gen},
             pdfauthor={Kari Eifler},
             pdfproducer={pdfLaTeX},
             pdfpagemode=None,
             bookmarksopen=true
             bookmarksnumbered=true]{hyperref}

\usepackage{tikz}
\usetikzlibrary{arrows,calc,decorations.pathreplacing,decorations.markings,intersections,shapes.geometric,through,fit,shapes.symbols,positioning,decorations.pathmorphing}

\usepackage{braket}

\usepackage[amsmath,thmmarks,hyperref]{ntheorem}
\usepackage{cleveref}

\crefname{section}{Section}{Sections}
\crefformat{section}{#2Section~#1#3} 
\Crefformat{section}{#2Section~#1#3} 

\crefname{subsection}{\S}{\S\S}
\crefformat{subsection}{#2\S#1#3} 
\Crefformat{subsection}{#2\S#1#3} 

\theoremstyle{plain}

\newtheorem{lemma}{Lemma}[section]
\newtheorem{proposition}[lemma]{Proposition}
\newtheorem{corollary}[lemma]{Corollary}
\newtheorem{theorem}[lemma]{Theorem}

\theoremstyle{nonumberplain}
\newtheorem{theoremN}{Theorem}

\theoremstyle{plain}
\theorembodyfont{\upshape}
\theoremsymbol{\ensuremath{\blacklozenge}}

\newtheorem{definition}[lemma]{Definition}
\newtheorem{example}[lemma]{Example}
\newtheorem{remark}[lemma]{Remark}

\crefname{definition}{definition}{definitions}
\crefformat{definition}{#2definition~#1#3} 
\Crefformat{definition}{#2Definition~#1#3} 

\crefname{ex}{example}{examples}
\crefformat{example}{#2example~#1#3} 
\Crefformat{example}{#2Example~#1#3} 

\crefname{remark}{remark}{remarks}
\crefformat{remark}{#2remark~#1#3} 
\Crefformat{remark}{#2Remark~#1#3} 

\crefname{convention}{convention}{conventions}
\crefformat{convention}{#2convention~#1#3} 
\Crefformat{convention}{#2Convention~#1#3}

\crefname{lemma}{lemma}{lemmas}
\crefformat{lemma}{#2lemma~#1#3} 
\Crefformat{lemma}{#2Lemma~#1#3} 

\crefname{proposition}{proposition}{propositions}
\crefformat{proposition}{#2proposition~#1#3} 
\Crefformat{proposition}{#2Proposition~#1#3} 

\crefname{corollary}{corollary}{corollaries}
\crefformat{corollary}{#2corollary~#1#3} 
\Crefformat{corollary}{#2Corollary~#1#3} 

\crefname{theorem}{theorem}{theorems}
\crefformat{theorem}{#2theorem~#1#3} 
\Crefformat{theorem}{#2Theorem~#1#3} 

\crefname{assumption}{assumption}{Assumptions}
\crefformat{assumption}{#2assumption~#1#3} 
\Crefformat{assumption}{#2Assumption~#1#3} 

\crefname{equation}{}{}
\crefformat{equation}{(#2#1#3)} 
\Crefformat{equation}{(#2#1#3)}

\theoremstyle{nonumberplain}
\theoremsymbol{\ensuremath{\blacksquare}}

\newtheorem{proof}{Proof.}

\newcommand{\A}{{\mathcal A}}
\newcommand{\B}{{\mathcal B}}
\newcommand{\C}{{\mathcal C}}

\newcommand{\G}{{\mathcal G}}
\renewcommand{\H}{{\mathcal H}}

\newcommand{\M}{{\mathcal M}}
\newcommand{\N}{{\mathcal N}}
\renewcommand{\O}{{\mathcal O}}

\newcommand{\R}{{\mathcal R}}
\renewcommand{\S}{{\mathcal S}}

\newcommand{\V}{{\mathcal V}}
\newcommand{\W}{{\mathcal W}}

\newcommand{\bC}{{\mathbb C}}

\newcommand{\bN}{{\mathbb N}}

\newcommand{\Irr}{\operatorname{Irr}}
\newcommand{\id}{\operatorname{id}}

\newcommand{\Mor}{\operatorname{Mor}}

\newcommand{\Rep}{\operatorname{Rep}}

\newcommand{\spn}{\operatorname{span}}

\newcommand{\upchi}{{\raise.35ex\hbox{$\chi$}}}

\title{Non-local games and quantum symmetries of quantum metric spaces}
\author{Kari Eifler\footnote{Department of Mathematics, Texas A\&M University,  \url{keifler@math.tamu.edu}}}

\begin{document}

\date{}
\maketitle

\begin{abstract}
We generalize Banica's construction of the quantum isometry group of a metric space to the class of quantum metric spaces in the sense of Kuperberg and Weaver. We also introduce quantum isometries between two quantum metric spaces, and we show that if a pair of quantum metric spaces are algebraically quantum isometric, then their quantum isometry groups are monoidally equivalent. Motivated by the recent work on the graph isomorphism game, we introduce a new two-player nonlocal game called the metric isometry game, where players can win classically if and only if the metric spaces are isometric. Winning quantum strategies of this game align with quantum isometries of the metric spaces.
\end{abstract}

%\tableofcontents

%%%%%%%%%%%%%%%%%%%%%%%%%%%%%%%%%%%%

\section{Introduction} \label{sec:introduction}

Non-local games provide a useful framework for exhibiting the power of quantum entanglement and have recently received significant attention. First proposed by physicist John Stewart Bell in the 1960's, a non-local game is played cooperatively by Alice and Bob against a referee; Alice and Bob may communicate prior to gameplay but are no longer able to communicate once a round begins. The two players may have access to a shared entangled quantum state and measurements performed on the entangled physical system allow them to correlate their answers to the referee in a way they would not be able to do classically (\cite{gilles}). The quantum correlations describing their behaviour can be modelled by quantum mechanics and there are various different mathematical models ($loc$, $q$, $qs$, $qc$) describing the outcome of a quantum experiment. \Cref{sec:preliminaries} recalls definitions and results regarding non-local games and their classical and quantum strategies along with an overview of compact quantum groups.

Work in quantum information theory has led to quantum versions of many concepts in classical mathematics. First introduced in \cite{atserias} by taking two finite graphs $X$ and $Y$, the graph isomorphism game $Iso(X,Y)$ has a winning classical strategy if and only if the two graphs are isomorphic. A natural question to ask is whether such a game exists for weighted graphs, or for other classical structures.

We define a metric isometry game, $Isom(X,Y)$, for two finite metric spaces $(X,d_X)$ and $(Y,d_Y)$. This game has inputs and outputs that are the disjoint union of the set of points in $X$ and $Y$ and has a winning classical strategy if and only if the two metric spaces are isometric. Furthermore, we will show that winning quantum strategies of this game align with quantum isometries of the metric spaces. The metric isometry game also has a close connection to the weighted graph isomorphism game, explained in \Cref{sec:metricisometrygame}.

Synchronous games form a special class of games where Alice and Bob share a set of questions and answers, and within a given round if both players receive the same question, they must produce the same answer. A bi-synchronous game has the additonal restriction that the only way both players may produce the same answer within a given round and win is if they were given the same question. The metric isometry game is a new example of a bi-synchronous (and therefore also synchronous) game.

Each synchronous game $G$ has a $*$-algebra $\A(G)$, defined by generators and relations, that is associated to it. The representation theory of the game $*$-algebra gives information about the existence of perfect strategies for each of the mathematical models listed above (\cite{helton, kim}). For the game $G=Isom(X,Y)$ then $\A(Isom(X,Y))$ is a non-commutative analogue to the function algebra of the space of isometries $X \rightarrow Y$. We say that the two metric spaces are algebraically quantum isometric if $\A(Isom(X,Y)) \neq 0$. It was shown in \cite{helton} that there exist games $G$ for which the $*$-algebra $\A(G)$ may be non-zero even if this algebra has no $C^*$-representations, and in particular, no perfect quantum strategies. In contrast, the graph isomorphism game played with two graphs has the property that if the game $*$-algebra is nontrivial, then it has a nontrivial $C^*$-representation. One motivation for defining the metric isometry game is that it provides another class of examples which exhibit this same phenomenon exhibited for the graph isomorphism game.

In \cite{kw}, Kuperberg and Weaver define a non-commutative analogue of a metric space, called a $W^*$-quantum metric space, which we introduce in \Cref{sec:qmetricspaces}. A $W^*$-quantum metric space is a one-paramenter family of weak$^*$-closed operator systems $\V = \{ \V_t\}_{t \geq 0}$. The intuition is that the $\V_t$ is a non-commutative analogue of pairs of points $(x,y)$ whose distance is at most $t$.

Given a finite metric space, we recall that the isomery group is a natural subgroup of the permutation group $S_n$. Specifically, the isometry group is the subgroup of $S_n$ satisfying the relations $\sigma D = D \sigma$ where $D$ is the distance matrix for the metric space and $\sigma$ is a permutation in the symmetry group. Similarly, in \cite{banica}, Banica defined the quantum isometry group of a finite metric space in a similar way: the quantum isometry group is a quantum subgroup of the quantum permutation group, defined as the quotient of the quantum permutation group by relations mimicking the classical case. In \Cref{sec:isometrygroup}, we generalize Bancia's definition to (possibly infinite) $W^*$-quantum metric spaces and show that the universal object defining the quantum isometry group exists in the finite dimensional case and agrees with Banica's definition.

In \Cref{sec:qisoqmetrics}, we define quantum isometrires between two $W^*$-quantum metric spaces, which generalizes isometries between classical metric spaces. We utilize the techniques in \cite{ours} to prove the following result:

\begin{theoremN}
Consider two quantum metric spaces, and suppose the quantum isometry group between the two quantum metric spaces is non-zero. Then the two quantum isometry groups corresponding to the two quantum metric spaces are monoidally equivalent.
\end{theoremN}

If we look at this restricted to the case of classical metric spaces, we have the following result:

\begin{theoremN} \label{thrm:qc}
Two classical metric spaces are algebraically quantum isometric if and only if the graph isomorphism game has a perfect quantum-commuting (qc)-strategy.
\end{theoremN}

\subsection*{Acknowledgements}

The author is indebted to her PhD supervisor, Michael Brannan, for many fruitful discussions and his guidance on this project. The author was partially supported by NSF grant DMS-2000331.

\section{Preliminaries} \label{sec:preliminaries}

\subsection{Notation}

When referring to tensor products, we use the symbol $\otimes$ to denote the tensor product of Hilbert spaces or the minimal tensor product of $C^*$-algebras. We use the symbol $\overline{\otimes}$ to denote the normal spatial tensor product of von Neumann algebras.
We use the standard leg numbering notation for linear operators on tensor products of vector spaces.

\subsection{Non-local games}

\subsubsection{Games and strategies} 

A {\it two-player non-local game} is a tuple $\G = (I_A, I_B, O_A, O_B, \lambda)$ where $I_A, I_B, O_A, O_B$ are finite sets representing the inputs and outputs for Alice and Bob and 

\[ \lambda: I_A \times I_B \times O_A \times O_B \rightarrow \{ 0,1 \} \]

is a {\it rule function}. The game is played coopertively by two players, Alice and Bob, against a referee. The game rules are known by all before the game begins, and Alice and Bob may agree on a strategy before beginning to play the game. While the game is being played however, Alice and Bob may no longer communicate and can only rely on the strategy they agreed upon.

A single round of the game consists of the referee giving Alice an input (question) $v$ from her set of possible inputs $I_A$, and giving Bob an input $w$ from his set of possible inputs $I_B$. Without communicating, Alice and Bob reply with outputs (answers) $a \in O_A$ and $b \in O_B$, respectively. They win the round if $\lambda(v,w,a,b) = 1$ and lose the round otherwise. Alice, Bob, and the referee play repeated rounds, and their goal is to win each round.

A game is called {\it synchronous} provided that the two players input sets are the same ($I = I_A = I_B$), as are their output sets ($O = O_A = O_B$), and the rule function satisfies

\[ \lambda(v,v,a,b) = \begin{cases} 0 & a\neq b \\ 1 & a=b \end{cases} \qquad \forall v \in I. \]

Another way to say this is that when Alice and Bob receive the same input, in order to win they must produce the same output. We call a game {\it bi-synchronous} as in \cite{bisynchronous} provided that the game is both synchronous and 

\[ \lambda(v,w,a,a) = \begin{cases} 0 & v \neq w \\ 1 & v=w \end{cases} \qquad \forall a \in O. \]

The strategies that Alice and Bob may utilize break into two categories: either a deterministic strategy or a random strategy. A {\it deterministic strategy} is a pair of functions $h : I_A \rightarrow O_A$ and $k: I_B \rightarrow O_B$ which determine the answers Alice and Bob give to the referee. If they receive $(v,w) \in I_A \times I_B$ then they respond with $(h(v), k(w)) \in O_A \times O_B$. A deterministic strategy wins every round if and only if $\lambda (v,w,h(v),k(w)) = 1$ for all $(v,w) \in I_A \times I_B$. We call such a strategy a perfect deterministic strategy. Given a synchronous game, a {\it perfect deterministic strategy} must satisfy $h=k$.

A {\it random strategy} or {\it probabilistic strategy} is characterized by the fact that on different rounds of the game, Alice and Bob may produce different outputs given the same input pair $(v,w)$. The idea is that even though there might not exist a perfect deterministic strategy to win the game, the players may improve their chance of winning the game by sampling their outputs from some joint probability distribution. As an outsider to the game, one may observe multiple rounds of the game to obtain the conditional probability density $p(a,b|v,w)$ which describes their behavious and represents the probability that given inputs $(v,w) \in I_A \times I_B$ that Alice and Bob produce outputs $(a,b) \in O_A \times O_B$. It's clear that $0 \leq p(a,b|v,w) \leq 1$ and that given a fixed $(v,w) \in I_A \times I_B$, $\sum_{a \in O_A, b\in O_B} p(a,b|v,w) = 1$.

We call a random strategy {\it perfect} if Alice and Bob win each round with probability 1. That is, the strategy is perfect if $\lambda(v,w,a,b) = 0$ implies that $p(a,b|v,w)=0$ for any $(v,w,a,b) \in I_A \times I_B \times O_A \times O_B$.

Assuming different mathematical models, we may get different sets of conditional probabilities $p(a,b|v,w)$. Given $n$ inputs and $k$ outputs, we denote the set of conditional probability densities that belong to each of these sets by $C_t(n,k)$ satisfying

\[ C_{loc}(n,k) \subseteq C_q(n,k) \subseteq C_{qs}(n,k) \subseteq C_{qa}(n,k) \subseteq C_{qc}(n,k) \]

where {\it local (loc)}, {\it quantum (q)}, {\it quantum spatial (qs)}, {\it quantum approximate (qa)}, and {\it quantum commuting (qc)} correspond to the different models. Here, local (or classical) correlations arise when Alice and Bob utilize only a shared probability space while quantum strategies arise from the random outcomes of entangled quantum experiments. We refer the reader to \cite{kim,lupini} for a thorough investigation of the models.

It's known that for $n,k \geq 2$ $C_{loc}(n,k) \neq C_q(n,k)$. It was shown in \cite{prakash} that for $n \geq 5, k \geq 2$ we have $C_{qs}(n,k) \neq C_{qa}(n,k)$, and \cite{coladangelo} showed for $n \geq 5, k\geq 3$ then $C_q(n,k) \neq C_{qs}(n,k)$. In \cite{ji} it was shown there exists $n,k$ such that $C_{qa}(n,k) \neq C_{qc}(n,k)$ which also disproves Connes' embedding conjecture posed in \cite{connesembedding}.

We say that a game has a {\it perfect $t$-strategy} if it has a perfect random  strategy that belongs to one of these models, where $t$ is one of {\it loc, q, qs, qa,} or {\it qc}.

\subsubsection{The $*$-algebra of a synchronous game}

In this subsection, we recall the definition of the $*$-algebra of a synchronous game and summarize the results found in \cite{kim, lupini, Timmermann}.

\begin{definition} \label{defn:staralgebra}
The $*$-algebra of a synchronous game $\G$, $\A(\G)$, is defined as the quotient of the free $*$-algebra generated by $\{ e_{v,a} \mid v \in I, a \in O \}$ subject to the relations

\begin{itemize}
\item $e_{v,a} = e_{v,a}^*$
\item $e_{v,a} = e_{v,a}^2$
\item $1 = \sum_a e_{v,a}$
\item $e_{v,a}e_{w,b} = 0$ for all $v,w,a,b$ such that $\lambda(v,w,a,b) = 0$
\end{itemize}
\end{definition}

The generators $e_{v,a}$ represent the measurement operators for Alice while the algebraic relations above are forced upon us by the restrictions of a winning strategy -- from both the mathematical formalism of quantum mechanics together with the structure of the rule function. In particular, since our game is synchronous and so $\lambda(v,w,a,b) = \delta_{a,b}$ then if $a \neq b$ we have $e_{v,a}e_{v,b} = 0$. Note that this algebra may be zero, and in fact, we are specifically interested in the cases where this algebra is non-zero.

The following theorem proved in \cite{kim} shows that the representation theory of the game $*$-algebra is crucial to understanding the existence of a winning $t$-strategy of the game.

\begin{theorem} \label{thrm:reptheoryAG}
Let $\G = (I,O,\lambda)$ be a synchronous game. Then

\begin{itemize}
\item $\G$ has a perfect deterministic strategy if and only if $\G$ has a perfect loc-strategy if and only if there exists a unital $*$-homorphism from $\A(\G)$ to $\bC$.
\item $\G$ has a perfect $q$-strategy if and only if $\G$ has a perfect $qs$-strategy if and only if there exists a unital $*$-homomorphism from $\A(\G)$ to $B(\H)$ for some non-zero finite dimensional Hilbert space $\H$.
\item $\G$ has a perfect $qa$-strategy if and only if there exists a unital $*$-homomorphism of $\A(\G)$ into the ultrapower of the hyperfinite $II_1$-factor.
\item $\G$ has a perfect $qc$-strategy if and only if there exists a unital $C^*$-algebra $\C$ with a faithful trace and a unital $*$-homomorphism $\pi : \A(\G) \rightarrow \C$.
\end{itemize}
\end{theorem}

\begin{definition}
We say a synchronous game $\G$ has a {\it perfect $A^*$-strategy} if $\A(\G)$ is non-zero. We say $\G$ has a {\it perfect $C^*$-strategy} if there exists a unital $*$-homomorphism form $\A(\G)$ into $B(\H)$ for some non-zero Hilbert space $\H$.

In general, these strategies are not physical and there is no guarantee of a corresponding physical correlation.
\end{definition}

\subsection{Compact quantum groups}

We will now review the basics of compact quantum groups, their actions and representations. The reader may be referred to references \cite{woronowicz, NeshveyevTuset, woronowicz2, chirvasitu} for details.

\begin{definition} \label{defn:cqg}
A {\it compact quantum group} is a unital $C^*$-algebra $\A$ equipped with a unital $*$-homomorphism called {\it comultiplication} $\Delta: \A \rightarrow \A \otimes \A$ such that

\begin{itemize}
\item $(\Delta \otimes \id) \circ \Delta = (\id \otimes \Delta) \circ \Delta$ as homomorphisms (co-associativity)
\item the spaces $\spn \{ (a \otimes 1) \Delta(b) \mid a,b \in \A \}$ and $\spn \{ (1 \otimes a) \Delta(b) \mid a,b \in \A \}$ are dense in $\A \otimes \A$ (the cancellation property)
\end{itemize}
\end{definition}

Motivation for this definition comes from the example given by $\A = C(G)$, the space of all continuous complex functions on a fixed compact group $G$. Here, comultiplication $\Delta: C(G) \rightarrow C(G \times G) \cong C(G) \otimes C(G)$ is given by $(\Delta(f))(g,h) = f(g \cdot h)$ so $\Delta$ captures the group operation at the level of $C(G)$.

Conversely, every compact quantum group $(\A, \Delta)$ whose underlying $C^*$-algebra $\A$ is commutative is of the form $\A = C(G)$ for some compact group $G$ \cite{woronowicz}.

\begin{remark}
Based on this commutative example, we use the notation $\A = C(G)$ for general compact quantum groups.
\end{remark}

\begin{remark}
An equivalent way to write \Cref{defn:cqg} is to view compact quantum groups as special Hopf $*$-algebras (see Section \ref{subsecmonoidal} and \cite{dijkhuizen}), with the the structure maps of the underlying compact group giving rise to a number of unital homomorphisms with $\Delta$ as above.
\end{remark}

We look at a few examples of compact quantum groups which will be used later. First, we define a {\it magic unitary} over a unital $*$-algebra $\A$ to be an $n \times n$ matrix $U = [u_{ij}]_{i,j}$ with entries $u_{ij} \in \A$ which satisfies

\begin{itemize}
\item $u_{ij} = u_{ij}^* = u_{ij}^2$
\item $\sum_{i=1}^n u_{ij} = 1 = \sum_{j=1}^n u_{ji}$
\end{itemize}

In the case where $\A$ is the complex numbers, a magic unitary matrix is simply a permutation matrix.

\begin{example}
The {\it quantum permutation group} $S_n^+$ \cite{wang} is the compact quantum group $(\A,\Delta)$ where $\A = C(S_n^+)$ is the universal $C^*$-algebra generated by the entires of an $n \times n$ magic unitary matrix $u = [u_{ij}]$. Comultiplication is given by the formula $\Delta(u_{ij}) = \sum_k u_{ik} \otimes u_{kj}$.

If we were to instead consider the universal $C^*$-algebra generated by commuting entries of an $n \times n$ magic unitary matrix, we would get the function algebra $C(S_n)$ of the symmetry group $S_n$. Thus, we should view $C(S_n^+)$ as a non-commutative symmetry group of a finite set of $n$ points with no extra structure. There always exists a quotient map from $C(S_n^+)$ into $C(S_n)$, allowing us to realize $S_n$ as a subgroup of $S_n^+$.

It was shown that for $n \geq 4$, Wang showed in \cite{wang} that $S_n^+$ is non-commutative, that is, that even classical objects such as four points with no additional structure can have quantum symmetries unseen when restricting to classical groups.
\end{example}

\begin{example}
The {\it universal unitary quantum group} $U_F^+$ associated to a matrix $F \in GL_n(\bC)$ \cite{wang} is the universal $*$-algebra generated by the entries of a $n \times n$ matrix $u = [u_{ij}]$ for which $(1 \otimes F)[u_{ij}^*](1 \otimes F^{-1})$ is a unitary in $M_n(C(U_F^+))$. The comultiplication map $\Delta$ is defined the same as for $S_n^+$.
\end{example}

Let $G = (C(G), \Delta)$ be a compact quantum group and $\H$ a finite dimensional Hilbert space of dimension $n$. In general, a {\it representation of $G$} is an invertible element $v \in \B(\H) \otimes C(G)$ such that $(\id \otimes \Delta)(v) = v_{12}v_{13}$. If we fix an orthonormal basis $(e_j)$ for $\H$, then a representation $v$ corresponds to an invertible matrix $v=[v_{ij}] \in M_n(C(G))$ such that $\Delta(v) = \sum_{k=1}^n v_{ik} \otimes v_{kj}$. We call $v$ a {\it unitary representation} if it is unitary. Given an infinite dimensional Hilbert space $\H$, one can similarly define an infinite dimensional unitary representation to be some $u \in M(K(\H) \otimes C(G))$ such that $(\id \otimes \Delta)(u) = u_{12}u_{13}$. We refer the reader to \cite{NeshveyevTuset} for details.

Fix two representations $v \in B(\H_v) \otimes C(G)$ and $u \in B(\H_u) \otimes C(G)$. A {\it morphism} between $u$ and $v$ is a linear map $T: \H_u \rightarrow \H_v$ that satisfies $(T\otimes 1)u = v(T \otimes 1)$, and we let $\Mor(u,v)$ be the Banach space of all morphisms between $u$ and $v$. We call representations $u$ and $v$ {\it equivalent} if $\Mor(u,v)$ contains an invertible element. Two representations are said to be {\it irreducible} if $\Mor(u,u) = \bC$. The set of equivalence classes of irreducible representations is denoted $\Irr(G)$. It's easy to show that if $u$ is a unitary representation, then $\Mor(u,u)$ is a $C^*$-algebra. We may consider the {\it direct sum} $u \oplus v \in B(\H_u \oplus \H_v) \otimes C(G)$, the {\it tensor product} $u \otimes v := u_{12} v_{13} \in B(\H_u \otimes \H_v) \otimes C(G)$ and {\it conjugate representation} $\overline{u} := [u_{ij}^*] \in B(\overline{\H}) \otimes C(G)$.

The representation category of a compact quantum group $G$ is defined to be the category whose objects are equivalence classes of finite dimensional representations of $G$ and is denoted $\Rep(G)$. An interested reader can refer to \cite{NeshveyevTuset} for more details.

The fundamental theorem on finite dimensional representations of compact quantum groups is analogous to the classical case. It is stated as follows:

\begin{theorem}(\cite{woronowicz})
Let $G$ be a compact quantum group. Every finite dimensional representation of $G$ is equivalent to a unitary representation, and every finite dimensional unitary representation of $G$ is equivalent to a direct sum of irreducible representations.

$C(G)$ is densely linearly spanned by the matrix elements of irreducible unitary representations of $G$.
\end{theorem}

\begin{definition} \label{defn:monoidallyequivalent} (\cite{bichon2, bichon3})
We say that two compact quantum groups $G_1$ and $G_2$ are monoidally equivalent, written $G_1 \sim_{mon} G_2$, if there exists a bijection between the equivalence classes of irreducible representations given by $\varphi: \Irr(G_1) \rightarrow \Irr(G_2)$ together with linear isomorphisms

\[ \varphi : \Mor(u_1 \otimes \ldots \otimes u_n, v_1 \otimes \ldots \otimes v_m) \rightarrow \Mor(\varphi(u_1) \otimes \ldots \varphi(u_n), \varphi(v_1) \otimes \ldots \otimes \varphi(v_m)) \]

such that for any morphisms $S, T$

\begin{itemize}
\item $\varphi(1_{G_1}) = 1_{G_2}$ where $1_{G_i}$ is the trivial representation of $G_i$
\item $\varphi(S \circ T) = \varphi(S) \circ \varphi(T)$ whenever $S \circ T$ is well-defined
\item $\varphi(S^*) = \varphi(S)^*$
\item $\varphi(S \otimes T) = \varphi(S) \otimes \varphi(T)$
\end{itemize}
\end{definition}

A special class of compact quantum groups are the compact matrix quantum groups. A {\it compact matrix quantum group} is a compact quantum group $G$ with a finite dimensional representation $u$ such that $C(G) = C^*(u_{ij}, u_{ij}^*)$. We call $u$ a fundamental representation of $G$.

Most compact quantum groups are presented as compact matrix quantum groups: both examples above are examples of compact matrix quantum groups.

\section{The metric isometry game} \label{sec:metricisometrygame}

A {\it finite metric space} is a finite set $X$ equipped with a finite metric $d: X \times X \rightarrow [0,\infty)$. Throughout this paper, we assume all metric spaces are finite, unless stated otherwise. Given two finite metric spaces $(X,d_X)$ and $(Y,d_Y)$, we say the two metric spaces are {\it isometric} if there is a bijection $f$ between $X$ and $Y$ that preserves distances, that is, $d_X(x,x') = d_Y(f(x),f(x'))$ for all $x, x' \in X$. If they are isometric, we write $X \simeq Y$.

To define the Metric Isometry Game, $Isom(X,Y)$, we set the inputs/outputs to be $I = O = X \sqcup Y$. Suppose the inputs for Alice and Bob are $v$ and $w$ while the outputs are $a$ and $b$, respectively. The players win the round if all of the following are satisifed:

\begin{enumerate}
\item $v$ and $a$ are from different spaces
\item $w$ and $b$ are from different spaces
\item If $v$ and $w$ are from the same metric space, then $d_\bullet(v,w) = d_\bullet(a,b)$
\item If $v$ and $w$ are from different spaces, then $v=b$ if and only if $w = a$.
\end{enumerate}

Note that condition 3 along with the fact that $d(v,w) = 0$ iff $v = w$ forces the game $Isom(X,Y)$ to be bi-synchronous.

\subsection{Connections to the directed graph isomorphism game}

We may recast $Isom(X,Y)$ in terms of graphs, and connect it to the graph isomorphism game. The graph isomorphism game consideres two finite, simple, undirected graphs and was first introduced in \cite{atserias}, and later studied in \cite{lupini, ours}. It is then natural to consider a modification of the graph isomorphism game for weighted graphs.

From a metric space $(X,d)$, one can derive a weighted graph $G_X = (V(G),E(G), w)$ arising from the metric space. We let $V(G) = X$ and let $E(G) = X \times X$ so that $G$ is the complete graph on $|X|$ vertices. We set the weight of the edge between $x$ and $y$ to be $d(x,y) = d(y,x)$.

We note immediately that $d(x,x) = 0$ implies that the graph has no loops, $d(x,y)=d(y,x)$ for all $x,y \in X$ implies that the graph is undirected, and the triangle inequality $d(x,z) \leq d(x,y) + d(y,z)$ for all $x,y,z \in X$ implies that the ``cheapest'' way to get from $x$ to $z$ (or vice versa) is directly.

The adjacency matrix for this graph, $A_X = [a_{ij}^X]_{i,j\in X}$, is given by $a_{ij}=d(i,j)$. It is a symmetric matrix with zeros along the diagonal.

We will now define the {\it weighted graph isomorphism game}, an expansion of the well-studied graph isomorphism game and show it is analogous to the Metric Isometry Game described earlier. We start the game with two simple, weighted graphs, $G$ and $H$.

\begin{definition} \label{def:mincompletegraph}
From a simple weighted graph $G$, we define the {\it minimum complete graph} $G'$ to be the complete graph on the same vertices with weight between vertices in $G'$ to be the cheapest path weight between the two vertices in $G$.
\end{definition}

From graphs $G$ and $H$, we obtain the minimum complete graphs $G'$ and $H'$. We set the inputs and outputs for the weighted graph isomorphism game to be $I = O = V(G) \sqcup V(H) = V(G') \sqcup V(H')$.

The referee will give two inputs, $v$ and $w$ to the two players, respectively Alice and Bob. They will reply with the outputs $a$ and $b$. We say that the players win that round if the following criteria are satisfied:

\begin{enumerate}
\item $v$ and $a$ are from different graphs
\item $w$ and $b$ are from different graphs
\item If $v$ and $w$ are from the same graph, then $d_\cdot(v,w) = d_\cdot(a,b)$ where the distance is the minimum path length of the minimum complete graphs
\item If $v$ and $w$ are from different graphs, then $v=b$ if and only if $w = a$.
\end{enumerate}

The directed graph isomorphism game is a reformulation of the metric isometry game.

\begin{theorem} \label{thrm:isomxyisomorphic}
Take two metric spaces $(X,d_X)$ and $(Y,d_Y)$ and let $G'$ and $H'$ be the corresponding minimum complete graphs.

\begin{enumerate}
\item $(X,d_X)$ and $(Y,d_Y)$ are isometric
\item The Metric Isometry Game played on $(X,d_X)$ and $(Y,d_Y)$ has a winning classical strategy
\item The minimum complete graphs $G'$ and $H'$ are isomorphic
\item The Weighted Graph Isomorphism game for $G'$ and $H'$ has a winning classical strategy
\end{enumerate}

\end{theorem}

\begin{proof}
$(1) \Leftrightarrow (3)$ is straighforward to check.

$(1) \Rightarrow (2)$: Suppose $(X,d_X)$ and $(Y,d_Y)$ are isometric, and $\varphi : X \rightarrow Y$ is an isomorphism. If the player receives point $x \in X$, then they should respond with $\varphi(x)$ and similarly, if the player receives vertex $y \in Y$, then they should respond with $\varphi^{-1}(y)$. This will win the $(X,d_X)-(Y,d_Y)$ metric isometry game and it is indeed a classical strategy.

$(2) \Rightarrow (1)$: Since the metric isometry game is a synchronous game, there must exist a winning deterministic strategy. Indeed, let Alice's answers be given by the function $f_A : X \sqcup Y \rightarrow X \sqcup Y$, and let Bob's answers be given by the function $f_B : X \sqcup Y \rightarrow X \sqcup Y$. Since the game is synchronous, the two functions must be equal and so call this function $f:= f_A = f_B$.

Note that the restriction $f|_{X} : X \rightarrow Y$ is an isomorphism from $X$ to a subset of $Y$. Similarly, the restriction $f|_{Y} : Y \rightarrow X$ is an isomorphism from $Y$ to a subset of $X$. This tells us that $X$ and $Y$ are isometric and that $f|_{X}$ and $f|_{Y}$ are isomorphisms.

We are left to show that $f|_{X} = f|_{Y}^{-1}$. That is, we want to show that for all $x \in X$, $x = f|_{Y} (f|_{X}(x)$. Consider the case where Alice receives $x \in X$ and Bob receives $f(x)$. The deterministic strategy dictates that Alice will respond with $f(x)$, and so because $y_A = x_B$ then the winning strategy criteria implies that $x_A = y_B$ and so Bob is forced to respond with $x$. This is true for all $x \in X$, and so $f|_{X} = f|_{Y}^{-1}$.

The proof of $(3) \Leftrightarrow (4)$ is the same as the proof above.
\end{proof}

\subsection{Game $*$-algebra of $Isom(X,Y)$}

We next want to know what the $*$-algebra of the metric isometry game is.

\begin{theorem}
Let $(X,d_X)$ and $(Y,d_Y)$ be metric spaces, each with size $n$. Then $\A(Isom(X,Y))$ is generated by $4n^2$ self-adjoint idempotents $\{e_{z,w} \mid z,w \in X \sqcup Y \}$ satisfying

\begin{enumerate}
\item $e_{x,x'} = 0$ for all $x,x' \in X$ and $e_{y,y'} =0$ for all $y,y' \in Y$
\item $e_{x,y}^2 = e_{x,y}^* = e_{x,y}$ for all $x \in X, y \in Y$
\item for $x \in X$ and $y \in Y$, $e_{x,y} = e_{y,x}$
\item $\sum_{y \in Y} e_{x,y} =1$ for all $x \in X$
\item $\sum_{x \in X} e_{x,y} =1$ for all $y \in Y$
\item $e_{x,y}e_{x,y'}=0$ for all $y \neq y'$
\item $e_{x,y}e_{x',y}=0$ for all $x \neq x'$
\item for any $x \in X$ and $y \in Y$, then

\[ \sum_{x' \in X} d_{X}(x,x') e_{x',y} = \sum_{y' \in Y} d_{Y}(y',y) e_{x,y'}\]
\end{enumerate}
\end{theorem}

\begin{proof}
The definition of the *-algebra gives us property (2). Similarly, it's quickly clear that (6) and (7) follow from the winning criteria of the game.

To see criteria (1), consider $x,x' \in X$. Then for all $z,w \in X \cup Y$, we have $\lambda(x,z,x',w)=0$. Therefore, for a fixed $z$, we have

\begin{align*}
e_{x,x'} &=  e_{x,x'} \left( \sum_{z \in X \cup Y} e_{w,z} \right) \\
&= \sum_{z \in X \cup Y} e_{x,x'} e_{w,z} \\
&= \sum_{z \in X \cup Y} \lambda(x,z,x',w) e_{x,x'}e_{w,z} \\
&= 0.
\end{align*}

So $e_{x,x'}=0$. Similarly, for $y,y' \in Y$, $e_{y,y'}=0$.

Criteria (4) and (5) follow easily: for any $x \in X$, then

\[ 1 = \sum_{z \in X \cup Y} e_{x,z} = \sum_{z \in Y} e_{x,y}. \]

To prove criteria (3), take some $x \in X$ and $y \in Y$ then

\[ e_{y,x} = e_{y,x} \left( \sum_{z \in Y} e_{x,z} \right) = \sum_{z \in Y} e_{y,x}e_{x,z} = \sum_{z \in Y} \lambda(y,x,x,z) e_{zx}e_{xz} = e_{yx}e_{xy}. \]

Similarly, $e_{xy}=e_{xy}e_{yx}$. So then

\[ e_{xy} = e_{xy}^* = (e_{xy}e_{yx})^* = e_{yx}^*e_{xy}^* = e_{yx}e_{xy} = e_{yx}. \]

Finally, for all $x \in X$ and $y \in Y$, and recalling that $\lambda(x',x,y,y')=1$ if and only if $d_X(x,x') = d_Y(y,y')$, we see that

\begin{align*}
\sum_{x' \in X} d_X(x,x') e_{x',y} &= \sum_{x' \in X} d_X(x,x') e_{x',y} \left( \sum_{y' \in Y} e_{x,y'} \right) \\
&= \sum_{x' \in X, y' \in Y} d_X(x,x') e_{x',y} e_{x,y'} \\
&= \sum_{x' \in X, y' \in Y} d_X(x,x') \lambda(x',x,y,y') e_{x',y}e_{x,y'} \\
&= \sum_{x' \in X, y' \in Y} d_Y(y',y) e_{x',y} e_{x,y'} \\
&= \left( \sum_{x' \in X} e_{x',y} \right) \sum_{y' \in Y} d_Y(y',y) e_{x,y'} \\
&= \sum_{y' \in Y} d_Y(y',y) e_{x,y'}
\end{align*}
\end{proof}

\begin{remark} \label{remark:AIsom}
Let $U=[e_{x,y}]_{x\in X, y\in Y}$. Then the above relations imply that $U$ is a magic unitary matrix and that $(1 \otimes D_X)U = U(1 \otimes D_Y)$. Equivalently, for the weighted graph isomorphism for minimal complete corresponding graphs $G', H'$ we have that $(1 \otimes A_{G'})U = U(1 \otimes A_{H'})$ where we note that the adjacency matrices for the graphs $G'$ and $H'$ are identical to the distance matrix for their corresponding metrics.

Thus, the game $*$-algebra $\A(Isom(X,Y)$ can be viewed a non-commutative analogue of the space of isometries from $X$ to $Y$.
\end{remark}

\begin{definition} \label{defn:qstrategy}
Motivated by \Cref{thrm:reptheoryAG}, for two metric spaces $(X,d_X)$ and $(Y,d_Y)$ we define

\begin{itemize}
\item $X \cong_q Y$ if and only if there exists $d$ and projections $E_{x,y} \in M_d$ such that $U = (E_{x,y})$ is a unitary in $M_n(M_d)$ and $(1\otimes D_X) U = U(1 \otimes D_Y)$.

\item $X \cong_{qa} Y$ if and only if there exists projections $E_{x,y} \in \R^\omega$ such that $U = (E_{x,y}) \in M_n(\R^\omega)$ is a unitary and $(1\otimes D_X) U = U(1 \otimes D_Y)$. Here, $R^\omega$ is the hyperfinite $II_1$ factor, and interested readers can learn more in \cite{AnPopa}.

\item $X \cong_{qc} Y$ if and only if there exists projections $E_{x,y}$ in some $C^*$-algebra $\A$ with a tracial state such that $U = (E_{x,y}) \in M_n(\A)$ is a unitary and $(1\otimes D_X) U = U(1 \otimes D_Y)$.

\item $X \cong_{C*} Y$ if and only if there exists projections $E_{x,y}$ in some Hilbert space $\H$ such that $U = (E_{x,y}) \in M_n(B(\H))$ is a unitary and $(1\otimes D_X) U = U(1 \otimes D_Y)$.
\end{itemize}

\end{definition}

\begin{remark}
Given two metric spaces $(X,d_X)$ and $(Y,d_Y)$ with corresponding minimal complete graphs $G'$, $H'$, then since the two game $*$-algebras $\A(Isom(X,Y))$ and $\A(Iso(G',H'))$ are the same, we can see that, using the notation from \cite{ours}, for any $t \in \{ loc, q, qa, qc, C^*, A^* \}$ we have that $X \cong_t Y$ if and only if $G' \cong_t H'$.
\end{remark}

\section{$W^*$-quantum metric spaces} \label{sec:qmetricspaces}

The definitions of a $W^*$-quantum metric space and the theorems that follow in this section were introduced in \cite{kw}. They have since been studied in \cite{swift}.

A non-commutative analogue of a metric space, $\V = \{\V_t\}_{t \geq 0}$, was defined in \cite{kw} using the language of von Neumann algebras, called a $W^*$-quantum metric. The intuition behind their definition is that each family $\V_t$ is a non-commutative analogue of pairs of points $(x,y)$ whose distance is at most $t$, while motivation for this definition comes primarily from the standard model of quantum error correction. The definition of a $W^*$-quantum metric is related to other models of quantum metric spaces: Connes notion of a spectral triple produces a $W^*$-quantum metric \cite{connes2}, and every $W^*$-metric produces Reiffel's Lipschitz seminorm \cite{rieffel2}.

\begin{definition} \label{defn:kuperbergweaver} (\cite{kw}, Definition 2.3)
A {\it $W^*$- quantum metric} on a von Neumann algebra $\M \subseteq B(\H)$ is a one-parameter family of weak* closed operator systems $\V_t \subseteq B(\H)$, $t \in [0,\infty)$ such that

\begin{enumerate}
\item $\V_s \V_t \subseteq \V_{s+t}$ for all $s,t \geq 0$
\item $\V_t = \cap_{s>t}\V_s$ for all $t \geq 0$
\item $\V_0 = \M'$ where $\M'$ is the commutant of $\M$ inside $B(\H)$
\end{enumerate}

We say a {\it $W^*$-quantum metric space} is the pair $(\M, \H,\{\V_t\}_{t \geq 0})$ of a von Neumann algebra $\M \subseteq B(\H)$ together with a $W^*$-quantum metric $\{\V_t\}$.
\end{definition}

It is easy to see that the $\V_t$ are nested. It can also be seen that the first and third condition implies $\V_0$ is a von Neumann algebra.

Given a (possibly infinite) metric space $(X,d)$, we can view the classical metric space as an example of a $W^*$-quantum metric on an abelian von Neumann algebra. We take the von Neumann algebra $\M = \ell^\infty(X)$ of bounded multiplication operators on $\ell^2(X)$ and define $\{ \V_t^X \}$ by

\begin{align*}
\V_t^X &= \overline{\spn}^{wk*} \{ V_{xy} \in \B(\ell^2(X)) \mid d(x,y) \leq t \} \\
&= \{ A \in B(\ell^2(X)) \mid \langle Ae_y,e_x \rangle =0 \text{ if } d(x,y)>t \}
\end{align*}
where $V_{x,y} \in B(\ell^2(X))$ is the rank one operator $V_{xy}: g \mapsto \langle g,e_y\rangle e_x$ and $\{e_x\}_{x \in X}$ is the standard orthonormal basis on $\ell^2(X)$.

\begin{proposition}  \label{prop:kwmetric} (\cite{kw}, Proposition 2.5.)
The construction above gives us a $W^*$-quantum metric space.
\end{proposition}

Conversely, if we have a $W^*$-quantum metric $\{\V_t \}$ on the commutative von Neumann algebra $\M = \ell^\infty(X)$, then we may set

\[ d(x,y) = \inf\{ t \mid \langle A e_y, e_x \rangle \neq 0 \text{ for some } A \in \V_t \} \]

to obtain a metric on $X$. Thus, we have obtained a correspondence between $W^*$-quantum metrics on abelian von Neumann algebras and classical metric spaces.

To motivate \Cref{defn:kuperbergweaver}, given a classical metric space $(X,d)$ we may look at the family of relations given by $R_t = \{ (x,y) \in X \times X \mid d(x,y) \leq t \}$. There is then the following correspondence between a classical metric space, this family of relations, and any quantum metric space as defined above:

\begin{align*}
d(x,x)&=0 &\leftrightarrow&& R_0 \text{ is the }&\text{diagonal relation} &\leftrightarrow&&  I &\in \M' = \V_0 \\
d(x,y) &= d(y,x) &\leftrightarrow&& R_t &= R_t^T &\leftrightarrow&& \V_t &= \V_t^* \\
d(x,z) &\leq d(x,y) + d(y,z) &\leftrightarrow&& R_s R_t &\subseteq R_{s+t} &\leftrightarrow&& \V_s\V_t &\subseteq \V_{s+t} 
\end{align*}

where $R_t^T$ denotes the transpose of the relation, that is, $(x,y) \in R_t$ if and only if $(y,x) \in R_t^T$. We can also note here that the relations $R_t$ are nested, as are the $\V_t$. From a family of relations $\{R_t\}_{t \geq 0}$ with the properties above, we can see that the relations define a unique metric $d(x,y) = \inf \{t \mid (x,y) \in R_t \}$.

\begin{example}
One can obtain a $W^*$-quantum metric from a classical graph $G = (V(G), E(G))$. If $|V(G)| = n$ then we may equip the space $V(G)$ with the shortest path metric coming from the graph. We then set

\[ \V_1 = \spn \{ E_{ij} \mid i=j \text{ or } i \text{ is adjacent to } j \} \subseteq M_n(\bC) = B(\ell^2(V(G))) \]

where $E_{ij} \in M_n(\bC)$ is the matrix of all zeros with a one in the $(i,j)$ entry. We can then set the larger sets to be

\[ \V_k = \V_1^k = \spn \{ A_1\ldots A_k \mid A_1,\ldots, A_k \in \V_1 \} \subseteq M_n(\bC). \]

It is not hard to check that this gives us a $W^*$-quantum metric on the von Neumann algebra $M_n(\bC) = B(\ell^2(V(G)))$.
\end{example}

A similar argument holds for the class of quantum graphs, an operator space generalization of classical graphs. We first define a quantum graph.

\begin{definition} \label{defn:qgraph}
A {\it quantum graph}, as defined in \cite{weaver,weaver2}, is a triple $(\S, \M, M_n)$ where $\M$ is a non-degenerate von Neumann algebra and $\M \subseteq M_n$, $\S \subseteq M_n(\bC)$ is an operator system and $\S$ is an $\M' - \M'$-bimodule with respect to matrix multiplication.
\end{definition}

\begin{remark}
An equivalent definition of a quantum graph has been defined in \cite{musto1}: a {\it quantum set} is a pair $X = (\B, \psi_X)$ where $\B$ is a finite dimensional $C^*$-algebra and $\psi_X : \B \rightarrow \bC$ is a faithful trace equipped with the multiplication map $m_X : \B \otimes \B \rightarrow \B$ and unit map $\eta_X : \bC \rightarrow \B$.

For $\delta > 0$, we call the state $\psi_X$ a {\it $\delta$-form} if $m_X m_X^* = \delta^2 \id$ where the adjoint is taken with respect to the Hilbert space structure on $\B$ coming from the GNS construction with respect to $\psi_X$. We label this Hilbert space $L^2(X)$.

We may then equip the quantum set with a linear map $A_X : L^2(X) \rightarrow L^2(X)$ which satisfies 

\begin{enumerate}
\item $m_X(A_X \otimes A_X) m_X^* = \delta^2 A_X$
\item $(\id \otimes \eta_X^* m_X) (\id \otimes A_X \otimes \id)(m_X^* \eta_X \otimes \id) = A_X$
\item $m_X (A_X \otimes \id) m_X^* = \delta^2 \id$
\end{enumerate}

We call such a triple $X = (\B, \psi_X, A_X)$ a {\it quantum graph}.

To see the connection between these two definitions, we fix the tracial $\delta$-form $\psi$ and we can set $\M = \B \subseteq B(L^2(X))$. We then set the $\M'-\M'$ bimodule $\S$ to be $P(B(L^2(X))$ where $P$ is the projection mapping the operator $T \in B(L^2(X))$ to $\delta^{-2} m_X (A_X \otimes T) m_X^*$. However, the relation between the two definitions is not one-to-one, as two distinct quantum graphs $(\B,\psi_X,A_X)$ in the sense of \cite{musto1} can yield the same $\M'-\M'$-bimodule $\S$.
\end{remark}

\begin{example} \label{ex:qgraph}
Our goal is to obtain a $W^*$-quantum metric from a quantum graph $X = (\S,\M,M_n)$. We set $\V_0 = \M'$, and if we assume $\S$ is non-reflexive, then $\V_1 = \S$ is orthogonal to $\M'$, that is, $\V_1 \subseteq \M'^\perp$.

Once we have $\V_1$, then we may set $\V_k = \V_1^k$ as before.

This connects to the quantum adjacency matrix definition: for operators $T$, 
we may consider the compression $m(A_X \otimes T)m^*$ which classically corresponds to the Schur multiplication $A_X \cdot T$. When we think of Schur multiplication by the adjacency matrix, it produces $\V_1$, that is, it kills all matrix units that aren't adjacent.
\end{example}

\section{Quantum isometry group of $W^*$-quantum metric spaces} \label{sec:isometrygroup}

Symmetries of a structure are viewed as transformations which preserve the relevant properties of that structure, while quantum symmetries are the non-commutative analogue of symmetries. The aim of this section is to generalize Banica's construction of the quantum isometry group for classical metric spaces to the class of $W^*$-quantum metric spaces. We define the quantum isometry group of the $W^*$-quantum metric spaces, answering the question which has been asked in \cite{goswami1}.

We first recall Banica's quantum isometry group for finite metric spaces.

\begin{definition} \label{defn:banica}(\cite{banica})
The {\it quantum isometry group} of a finite metric space $(X,d)$ is defined to be $\A = C(G^+(X,d))$ where $G^+(X,d) = (\A,\Delta)$ is the quotient of $C(S_n^+)$ by the ideal generated by the relations $UD = DU$, where $D = [d(x,y)]_{x,y \in X}$ is the distance matrix. That is,

\[ C(G^+(X,d)) = C(S_n^+) / \langle UD=DU \rangle. \]

Comultiplication is given by $\Delta : \A \rightarrow \A \otimes \A$ which maps $u_{ij} \mapsto \sum_k u_{ik} \otimes u_{kj}$.
\end{definition}

\subsection{Actions of a quantum group on a $W^*$-quantum metric}

For a compact quantum group $G$, we denote $C_r(G)$ to be the corresponding reduced $C^*$-algebra, that is, the image of $C(G)$ under the GNS representation $\pi_h: C(G) \rightarrow B(L^2(G))$. We equip $C_r(G)$ with a comultiplication $\Delta$. We denote by $L^\infty(G)$ to be the von Neumann algebra generated by $C_r(G)$ in $B(L^2(G))$ and the extension of the corresponding comultiplication $L^\infty \rightarrow L^\infty(G) \overline{\otimes} L^\infty(G)$ will be denoted $\Delta_G$. An interested reader can see \cite{kustermansvaes, kustermans, NeshveyevTuset} for more information.

\begin{definition}
Given a compact quantum group $G$ and a von Neumann algebra $\M$, an {\it action of $G$ on $\M$} is a normal injective unital $*$-homomorphism $\alpha : \M \rightarrow \M \overline{\otimes} L^\infty(G)$ that satisfies the {\it action equation}

\[ (\alpha \otimes \id_\M) \circ \alpha = (\id_{L^\infty(G)} \otimes \Delta_G) \circ \alpha. \]
\end{definition}

\begin{definition} \label{def:invariantstate}
Let the compact quantum group $G$ act on a von Neumann algebra $\M$ by $\alpha: \M \rightarrow \M \overline{\otimes} L^\infty(G)$. We call a state $\psi$ {\it $\alpha$-invariant} if $(\id \otimes \psi)\alpha = \psi(\cdot) 1$.
\end{definition}

It's known that we are guaranteed to have an invariant state in certain circumstances.

\begin{proposition} \label{thrm:invariantstate}
Consider a von Neumann algebra $\M$ with a faithful state $\phi$ and a compact quantum group $G$. If the compact quantum group $G$ acts on $\M$ by $\alpha : \M \rightarrow \M \overline{\otimes} L^\infty(G)$, then there exists a (not necessarily unique) $\alpha$-invariant state $\psi$ with $(\psi \otimes 1) \alpha(x) = \psi(x)1$.
\end{proposition}

The proof follows by letting $\psi (x) = (\phi \otimes h) \alpha(x)$ where $h$ is the Haar measure and one can view this as an average relative to the action $G \curvearrowright^\alpha \M$. Since $\phi$ is faithful, one can show $\psi$ is also faithful.

To motivate the next definition, we note that given a representation $U \in M(K(\H) \otimes C(G))$, we automatically get an action

\begin{equation} \label{eqn:alpha}
\begin{aligned}
\alpha: B(\H) &\rightarrow B(\H) \overline{\otimes} L^\infty(G) \\
T &\mapsto U^*(T \otimes 1) U.
\end{aligned}
\end{equation}

\begin{definition} \label{defn:unitarilyimplemented}
We begin with a $W^*$-quantum metric $(\V_t)$ on the von Neumann algebra $\M \subseteq B(\H)$. Let $G$ be a compact quantum group. Given a von Neumann algebraic action

\[ \alpha: \M \rightarrow \M \overline{\otimes} L^\infty(G) \]

we say the action is {\it unitarily implemented} if there exits a unitary representation $U \in M(K(\H) \otimes C_r(G))$ such that

\[ \alpha(x) = U^* (x\otimes 1)U \qquad x \in \M. \]
\end{definition}

It was shown in \cite{vaes1} that there always exists a unitarily implemented action from $\alpha: \M \rightarrow \M \overline{\otimes} L^\infty(G)$ for free if we have an invariant state.  By Prop \ref{thrm:invariantstate}, if we have the action $\alpha$ and fix any faithful state then we can obtain a faithful $\alpha$-invariant state. With respect to the GNS Hilbert space for $\M$, we can realize this as unitarily implemented by averaging the state with the Haar measure.

We may then naturally extend such a unitarily implemented action given by Equation \Cref{eqn:alpha} to an action on $B(\H)$.

\begin{definition} \label{defn:CQGacting}
Given a $W^*$-quantum metric space $(\M, \H, \V_t)$ with $\M \subseteq B(\H)$ and a compact quantum group $G$, we say that $G$ {\it acts} on the $W^*$-quantum metric space if there is a unitarily implemented, von Neumann algebraic action $\alpha: \M \rightarrow \M \otimes L^\infty(G)$ such that

\[ \alpha(\V_t) \subseteq \V_t \overline{\otimes}^{w*} L^\infty(G) \qquad \forall t \geq 0 \]
\end{definition}

\begin{definition} \label{defn:universal}
Given a possibly infinite dimensional $W^*$-quantum metric space $(\M,\H,\V_t)$, we say the universal compact quantum group $C(G^\V)$ acting on the quantum metric space which satisfies the following:

\begin{enumerate}
\item it is generated by a fundamental representation $\mathbb{U} \in M(K(\H) \otimes C^u(G))$ where $C^u(G)$ is the universal $C^*$-algebra associated to $G$

\item for any compact quantum group $C(G)$ acting on $(\M,\H,\V_t)$ with unitary representation $U \in M(K(\H)\otimes C_r(G))$, there exists a surjection $C^u(G^\V) \rightarrow C^u(G)$ which maps $\mathbb{U} \mapsto U$.
\end{enumerate}

We define the {\it quantum isometry group of the $W^*$-quantum metric space}, $G^{\V}$, to be the universal compact quantum group (if such a universal object exists) acting on the quantum metric space as in \Cref{defn:CQGacting}.
\end{definition}

It's not clear whether such a universal object exists in general. However, it can be shown that one exists in the finite dimensional case, which leads us into our next section.

\subsection{Finite dimensional case}

It is known that the quantum metrics do not depend on the choice of Hilbert space on which $\M$ is represented, and this result will be crucial to us.

\begin{theorem} \label{thrm:correspondence} (\cite{kw}, Theorem 2.4.)
Let $\H_1$ and $\H_2$ be Hilbert spaces and let $\M_1 \subseteq B(\H_1)$ and $\M_2 \subseteq B(\H_2)$ be isomorphic von Neumann algebras. Then any isomorphism induces an order preserving 1-1 correspondence between the quantum metrics on $\M_1$ and $\M_2$.
\end{theorem}

Therefore, for the following definition we may assume that we're representing 
$\V_t \subseteq B(L^2(\M))$ in the regular representation of $\M$.

\begin{definition} \label{defn:fdquantumautomorphism}
Let $\M \subseteq B(\H)$ be a finite dimensional von Neumann algebra with its canonical trace fixed. The {\it quantum automorphism group of $\M$}, denoted $G_{aut}$, is the universal compact quantum group with the following properties:

\begin{enumerate}
\item $C(G_{aut})$ is generated by the entries of a representation $U \in B(\H) \otimes C(G_{aut})$

\item By identifying $\H = \M$ as vector spaces, then we define the trace-preserving unital $*$-homomorphism $\delta$ on $\M$ as

\begin{equation} \label{eqn:delta}
\begin{aligned}
\delta : \M &\rightarrow \M \otimes C(G_{aut}) \\
e_j &\mapsto \sum_k e_k \otimes u_{kj}
\end{aligned}
\end{equation}

where $U = [u_{ij}]$ is the fundamental unitary representation.
\end{enumerate}

This $U$ is a unitary representation of the compact quantum group using the natural comultiplication $\Delta(u_{ij}) = \sum_{k=1}^n u_{ik} \otimes u_{kj}$ if and only if $\delta$ as defined above is an action.
\end{definition}

Since we assume that $\delta$ is a trace-preserving unital $*$-homomorphism, \cite{banica} shows that $U$ is automatically unitary.

Now, for any action $\alpha: \M \to \M \otimes C(G)$ that preserves the canonical trace, it was shown in \cite{vaes1} that for the Hilbert space $\H = L^2(\M)$, there always exists a unitary representation $V \in B(\H) \otimes C(G)$ that implements $\alpha$ via 
$\alpha(T) = V^*(T \otimes 1)V$ for each $T \in \M$. 

\begin{definition} \label{defn:fdquantumisometry}
Let $\M$ be a finite dimensional von Neumann algebra. Fix a $W^*$-quantum metric $(\V_t)_{t \geq 0}$, where by \Cref{thrm:correspondence} we may assume that the metric space is represented on the GNS Hilbert space, $(\V_t) \subseteq B(L^2(\M))$.

We define the {\it quantum isometry group} $C(G^\V)$ to be the quantum subgroup of $G_{aut}$ generated by $u_{ij}$ where the map $\delta$ in \Cref{eqn:delta} is a $\psi$-preserving $*$-homomorphism and the conjugation action $\alpha_\V$ given by 

\begin{equation} \label{eqn:alphaV}
\begin{aligned}
\alpha_\V: B(L^2(\M)) &\rightarrow B(L^2(\M)) \otimes C(G^\V) \\
T &\mapsto U(T \otimes 1) U^*
\end{aligned}
\end{equation}

leaves the $\V_t$ invariant, i.e. $\alpha_\V (\V_t) \subseteq \V_t \otimes C(G^\V)$ for all $t$.
\end{definition}

Here, $G^\V$ will be of Kac type. Indeed, $G^\V$ is a quantum subgroup of $G_{aut}$, the quantum automorphism group of a tracial von Neumann algebra, which is known to be of Kac type.

\begin{remark}
We may begin without any assumptions on the map $\alpha_\V$ given by equation \Cref{eqn:alphaV}, and consider only the action $\delta$ of $G$ on $\M$ as in equation \Cref{eqn:delta}. Then by the fact that $\delta$ is automatically unitarily implemented by equation \Cref{eqn:alphaV}, one can prove that $\V_0 = \M'$ is always preserved by the conjugation action $\alpha_\V$. That is, the condition that the conjugation action $\alpha_\V$ leaves $\V_t$ invariant at $t=0$ comes for free.
\end{remark}

\begin{remark}
We can show that the invariance in equation \Cref{eqn:alphaV} of the definition holds for both $\M'$ and $\M$. We know that $\M = L^2(\M) = \M' = J\M J$, where $J$ is modular conjugation.  At the level of the Hilbert space $\H$, the map $JTJ$ is just $T \mapsto T^*$.  From the definition of the action, we get that $\alpha_\V(\M) \subseteq \M \otimes B(\H)$.

Moreover, by noting that $\alpha_\V(Je_kJ) = \alpha(e_k^*) = \alpha_\V(e_k)^*$, then $\alpha_\V(\M') \subseteq \M' \otimes B(\H)$.  This shows that in equation \Cref{eqn:alphaV} of \Cref{defn:fdquantumisometry} at $t=0$, we still have the appropriate properties but the restriction of the map $T \mapsto U(T \otimes 1)U^*$ to $\V_{0} = \M'$ is not a *-homomorphism, but instead an anti-homomorphism.

In the case of classical metric spaces, $\M$ is commutative and thus we get that homomorphisms and antihomomorphisms are the same.
\end{remark}

\begin{proposition}
Let $(X,d)$ be a classical metric space, and consider the construction of the $W^*$-quantum metric space as in \Cref{prop:kwmetric}, with $\M = \ell^\infty(X)$ and $\H = \ell^2(X)$.

Then the quantum isometry group of the metric space is the same as the quantum isometry group of the corresponding $W^*$-quantum metric space, that is, $G^\V \cong G^+(X,d)$.
\end{proposition}

\begin{proof}
Banica's $(G^+(X,d),U)$ can be seen to satisfy the properties of $G^\V$ and so $G^+ < G^\V$.

To see the inclusion $G^\V < G^+$, by the properties defining the quantum automorphism group, we know the fundamental representation $U$ of $G^\V$ is a magic unitary.
Therefore, $G^\V < S_n^+$ where $n = |X|$. We need to check that $U(D \otimes 1)U^* = D$ follows from the invariance of $\alpha_\V$.

For $x,y \in X$ with $d(x,y) \leq t$, we know that $\alpha_\V(V_{xy}) \in \V_t^X$, and so we may write it as $\alpha_\V(V_{xy}) = \sum_{d(s,k)\leq t} V_{sk} \otimes x_{sk}$ for some $x_{sk} \in \C(G)$. By the definition of $\alpha_\V$, we see that $\alpha_\V(V_{xy}) = \sum_{s,k} V_{sk} \otimes u_{xs} u_{yk}$. Consider $a,b \in X$ with $d(a,b) > t$. We may start with

\[ 
\sum_{s,k, d(s,k) \leq t} V_{sk} \otimes x_{sk} = \sum_{s,k} V_{sk} \otimes u_{xs}u_{yk}
\]

We multiply by $(V_{aa} \otimes 1)$ on the left and by $(V_{bb} \otimes 1)$ on the right to obtain

\begin{equation} \label{eqn:work}
\sum_{s,k, d(s,k) \leq t} V_{aa}V_{sk}V_{bb} \otimes x_{sk} = \sum_{s,k} V_{aa}V_{sk}V_{bb} \otimes u_{xs} u_{yk}.
\end{equation}

We will use the fact that $V_{aa}V_{sk}V_{bb} = \delta_{s=a} \delta_{k=b} V_{ab}$. Thus, the sum on the left hand side of equation \Cref{eqn:work} equals zero since $d(s,k) = d(a,b) > t$. On the right hand side of equation \Cref{eqn:work}, the sum simply condenses down to $V_{ab} \otimes u_{xa} u_{yb}$. Therefore

\[ 0 = V_{ab} \otimes u_{xa}u_{yb}. \]

So indeed, if $d(x,y) \neq d(a,b)$ then $V_{ax}V_{by} = 0$ implying

\[ 0 = (V_{ax}V_{by})^* = V_{by}^*V_{ax}^* = V_{by}V_{ax}.\]

Similarly, by using the antipode or considering an analogous $\beta$-action, it can be shown that $V_{xa}V_{yb} = 0 = V_{yb}V_{xa}$.

We now claim that this is equivalent to $U(D\otimes 1) U^*= D$. Indeed, the $(i,j)$ entry of $U(D \otimes 1)U^*$ can be calculated as follows:

\begin{align*}
[U (D\otimes 1)U^*]_{ij} &= \sum_{k,s} d(k,s) u_{ik} u_{js} \\
&= \sum_{k,s, d(i,j)=d(k,s)} d(i,j) u_{ik} u_{js} \\
&= d(i,j) \sum_{k, s, d(i,j) = d(k,s)} u_{ik}u_{js} \\
&= d(i,j) \sum_{k,s} u_{ik}u_{js} \\
&= d(i,j) \left( \sum_k u_{ik} \right) \left( \sum_j u_{js} \right) \\
&= d(i,j) \cdot 1 = d(i,j)
\end{align*}

So then $G^{\V} < G^+$.
\end{proof}

\begin{remark}
Extensions of Banica's quantum isometry group on classical metric spaces were studied in \cite{quaegebeur}, where quantum isometry groups of quantum metric spaces in the framework of Rieffel is studied. Although both our definition and theirs agrees with Banica's definition in the classical sense, it would be interesting to further investigate the connection between the two extensions.
\end{remark}

\section{Quantum Isometries between Quantum Metric Spaces} \label{sec:qisoqmetrics}

In this section, we explore the quantum isometries between two $W^*$-quantum metric spaces. We restrict our attention in this section to the finite dimensional case, where interesting universal algebras are guaranteed to exist. This section uses techniques from \cite{ours}.

\begin{definition} \label{defn:qmetricsqisom}
Consider two finite dimensional quantum metric spaces $(\M_1, \H_1, \V_t)$ and $(\M_2, \H_2, \W_t)$ where the canonical traces on $\M_1$ and $\M_2$ are fixed and $\{e_j\}$ and $\{f_k\}$ are orthonormal bases for $\M_1$ and $\M_2$. Moreover, by \Cref{thrm:correspondence} we may assume $\H_i = L^2(\M_i)$.

We define $C(G^{\V,\W})$ to be the universal $C^*$-algebra generated by the coefficients of a unitary $P = [p_{ij}] \in C(G^{\V,\W}) \otimes B(\H_1, \H_2)$ with relations giving a unital $*$-homomorphism

\begin{align*}
\delta_{\V,\W}: \M_1 &\rightarrow \M_2 \otimes C(G^{\V,\W}) \\
e_j &\mapsto \sum_k f_k \otimes p_{kj}
\end{align*}

and ensuring the conjugation map given by

\begin{align*}
\alpha_{\V,\W} : B(\H_1) &\rightarrow B(\H_2) \otimes C(G^{\V,\W}) \\
T &\mapsto P(T\otimes 1)P^*
\end{align*}

satisfies $\alpha_{\V,\W}(\V_t) \subseteq \W_t \otimes C(G^{\V,\W})$.
\end{definition}

This definition satisfies two crucial criteria:

\begin{enumerate}
\item $G^{\V,\V} = G^\V$, as desired  

\item For classical metric spaces $(X,d_X), (Y,d_Y)$ and their corresponding $W^*$-quantum metric spaces $(\ell^2(X), \V_t)$ and $(\ell^2(Y), \W_t)$, we have $C(G^{\V,\W}) = \A(Isom(X,Y))$ 
\end{enumerate}

\begin{proposition}
Given two quantum graphs $X_1 = (\B_1, \psi_1, A_1)$ and $X_2 = (\B_2,\psi_2,A_2)$, we let $(\M,\V_t)$ and $(\N,\W_t)$ be the associated $W^*$-metric spaces.

If the two quantum graphs are quantum isometric, then so are their associated $W^*$-metric spaces.
\end{proposition}

\begin{proof}
If the quantum graphs $X_1$ and $X_2$ are quantum isomorphic, then there exists some $C^*$-algebra $\C$ and some unitary $P \in \C \otimes B(L^2(X_1), L^2(X_2))$ which intertwines the unit maps $\eta_{\B_i}$ and the multiplication maps $m_{\B_i}$ such that the map

\begin{align*}
\alpha_{12}: B(L^2(\B_1)) &\rightarrow B(L^2(\B_2)) \otimes \C \\
T &\mapsto P(T\otimes 1)P^*
\end{align*}

satisfies $P(A_1 \otimes 1) = (A_2 \otimes 1)P$.

For the associated operator systems $\S_i = \V_i$ defined in \Cref{ex:qgraph}, one can show that $\alpha_{12}(\S_1) \subseteq \S_2 \otimes \C$. It then immediately follows that $\alpha_{12}(\S_1^k) \subseteq \S_2^k \otimes \C$.
\end{proof}

\subsection{Monoidal equivalence and bigalois extensions} \label{subsecmonoidal}

We begin by viewing a compact quantum group as a Hopf $*$-algebra.

\begin{definition} \label{defn:hopf}
A {\it Hopf $*$-algebra} is a pair $(\A, \Delta)$ where $\A$ is a unital $*$-algebra and $\Delta: \A \rightarrow \A \otimes \A$ is a unital $*$-homomorphism that satisfies $(\Delta \otimes \id) \circ \Delta = (\id \otimes \Delta) \circ \Delta$ and for some $*$-homomorphism $\epsilon : \A \rightarrow \bC$ satisfying $(\epsilon \otimes \id) \Delta(a) = a = (\id \otimes \epsilon) \Delta(a)$ and for some anti-homomorphism $S : \A \rightarrow \A$ such that $m(S \otimes \id) \Delta(a) = \epsilon(a) 1 = m(\id \otimes S) \Delta(a)$ where $m: \A \otimes \A \rightarrow \A$ is the multiplication map.
\end{definition}

\begin{definition}
Given a compact quantum group $G$, we define the Hopf $*$-algebra $\O(G)$ to be the $*$-algebra generated by the coefficients $u_{ij}$ of the fundamental representation $U = [u_{ij}]$.
\end{definition}

\begin{theorem}
For any compact quantum group $G$, the pair $(\O(G),\Delta)$ is a Hopf $*$-algebra.
\end{theorem}

\begin{definition}
We say that the quantum metric spaces $(\M_1, \H_1,\V_t)$ and $(\M_2, \H_2, \W_t)$ are {\it $A^*$ quantum isometric} if $\O(G^{\V,\W}) \neq 0$, and we write $(\M_1,\V_t) \cong_{A^*} (\M_2,\W_t)$. We say they are are {\it $C^*$ quantum isometric} if $\O(G^{\V,\W})$ has a $C^*$-representation, and we write $(\M_1,\V_t) \cong_{C^*} (\M_2,\W_t)$. Finally, we say they are {\it $qc$-quantum isometric} if $\O(G^{\V,\W})$ admits a tracial state and write $(\M_1,\V_t) \cong_{qc} (\M_2,\W_t)$, so that by the work of \cite{lupini} this is the same as the existing notation for classical metric spaces.
\end{definition}

If we set $\A = \O(G)$ viewed as a Hopf $*$-algebra, then we call a unital $*$-algebra $Z$ a {\it left $\A$ $*$-comodule algebra} if it is equipped with a unital $*$-homomorphism $\alpha : Z \rightarrow \A \otimes Z$ satisfying

\[ (\id \otimes \alpha) \circ \alpha = (\Delta \otimes \id ) \circ \alpha \qquad (\epsilon \otimes \id ) \circ \alpha = \id \]

In the same vein, we call a untial $*$-algebra $Z$ a {\it right $\A$ $*$-comodule algebra} is equipped with a unital $*$-homomorphism $\beta : Z \rightarrow Z \otimes \A$ satisfying

\[ (\beta \otimes \id) \circ \beta = (\id \otimes \Delta) \circ \beta \qquad (\id \otimes \epsilon) \circ \beta = \id. \]

We call a left $\A$ $*$-comodule algebra a {\it left $\A$ Galois extension} if the linear map

\begin{align*}
\kappa_l : Z \otimes Z &\rightarrow \A \otimes Z \\
\kappa_l(x \otimes y) &= \alpha(x) (1 \otimes y)
\end{align*}

is bijective. A right $\A$ $*$-comodule alebra is called a {\it right $\A$ Galois extension} if the corresponding linear map

\begin{align*}
\kappa_r : Z \otimes Z &\rightarrow Z \otimes \A \\
\kappa_r(x \otimes y) &= (x \otimes 1) \beta(y)
\end{align*}

is bijective. For Hopf $*$-algebras $\A$ and $\B$, we call the unital $*$-algebra $Z$ an {\it $\A-\B$ bigalois extension} if it is both a left $\A$ Galois extension and a right $\B$ Galois extension and for left and right comodule maps $\alpha$ and $\beta$, then

\[ (\id \otimes \beta) \alpha = (\alpha \otimes \id) \beta : Z \rightarrow \A \otimes Z \otimes \B. \]

We define a {\it state} on a unital $*$-algebra $Z$ to be a linear functional $\omega : Z \rightarrow \bC$ such that $\omega(1)=1$ and $\omega(z^*z) \geq 0$ for all $z \in Z$.

If $Z$ is an $\A-\B$ bigalois extension, then a state $\omega$ is called {\it left-invariant} if $(\id \otimes \omega)(z) = \omega(z)1_\A$ for all $z \in Z$. Similarly, it is called {\it right-invariant} if $(\omega \otimes \id)(z) = \omega(z)1_\B$ for all $z \in Z$. The state $\omega$ is called {\it bi-invariant} if it is both left and right-invariant.

\begin{theorem}\label{thrm:bichon} (\cite{bichon2,bichon3})
Consider two compact quantum groups $G_1$ and $G_2$. Then $G_1$ and $G_2$ are monoidally equivalent if and only if there exists a $\O(G_1)-\O(G_2)$-bigalois extension $Z$ equipped with a bi-invariant state $\omega$.

Moreover, $\omega$ is a tracial state if and only if both $G_1$ and $G_2$ are of Kac type (that is, the Haar states are tracial).
\end{theorem}

It was first shown in \cite{bichon2} that there is a simple criterian for a bigalois extension to admit such a bi-invariant state $\omega$. We will show this argument later in the next section.

\subsection{A $G^\V-G^\W$ bigalois extension}

Our next goal is to show that $\O(G^{\V,\W})$ admits a natural structure as a $G^\V-G^\W$ bigalois extension.

\begin{theorem} \label{thrm:bigalois_extension}
If $\O(G^{\V,\W})$ is non-zero, then there exists a $G^\V - G^\W$ bigalois extension.
\end{theorem}

\begin{proof}
\begin{comment}
We consider the comodule-algebra structure map

\[ \delta_\W: \H_2 \rightarrow \H_2 \otimes G^\W \]

By the universality of $\delta_{\V,\W}$ the composition

\[ (\delta_\W \otimes \id) \circ \delta_{\V,\W} : \H_1 \rightarrow \H_2 \otimes G^\W \otimes G^{\V,\W} \]

must factor as $(\id \otimes \alpha) \circ \delta_{\V,\W}$ for a unique $*$-algebra morphism

\[ \alpha: G^{\V,\W} \rightarrow G^\W \otimes G^{\V,\W} \]

simply given by $\alpha(p_{ij} = \sum_k u_{ik} \otimes p_{kj}$ where $u = [u_{ij}]$ is the fundamental representation of $G^\W$.

Similarly $G^{\V,\W}$ has a right $G^\V$ $*$-comodule algebra structure given by

\[ \beta: G^{\V,\W} \rightarrow G^{\V,\W} \otimes G^\V \qquad \beta(p_{ij}) = \sum_k p_{ik} \otimes v_{kj} \]

where $v = [v_{ij}]$ is the fundamental representation of $G^\V$.
\end{comment}
One can show that $\O(G^{\V,\W})$ is an $\O(G^\W)-\O(G^\V)$ bicomodule following the proof in Section 4 of \cite{ours}.

We define ``cocomposition'' $*$-morphisms

\begin{align*}
\gamma_\W : G^\W &\rightarrow G^{\V,\W} \otimes G^{\W,\V} \qquad& \gamma_\W(u_{ij}) &= \sum_k p_{ik} \otimes q_{kj} \\
\gamma_\V : G^\V &\rightarrow G^{\W,\V} \otimes G^{\V,\W} \qquad& \gamma_\V(u_{ij}) &= \sum_k q_{ik} \otimes p_{kj} \\
\end{align*}

where $q=[q_{ij}]$ is the matrix of generators of $G^{\W,\V}$.

We now have a two-object cocategory $\C$: the four algebras $G^\V, G^\W, G^{\V,\W}$ and $G^{\W,\V}$ are thought of as dual to ``spaces of morphisms'' between two objects ($x \mapsto x$ for $G^\V$, $x \mapsto y$ for $G^{\V,\W}$, etc) and the $\gamma$ maps are dual to morphism composition.

Next, we make $\C$ into a cogroupoid by defining coinversion maps

\begin{align*}
S_{\V,\W} : G^{\W,\V} &\rightarrow G^{\V,\W} \\
S_{\W,\V} : G^{\V,\W} &\rightarrow G^{\W,\V}
\end{align*}

Let $F = F_\V \in M_n$ and $G = F_\W \in M_m$ be matrices with the property that $Fe_i = e_i^*$ and similarly for $G$, so that $\overline{F} = F^{-1}$ and $\overline{G} = G^{-1}$. Then we have the following involutivity of morphisms on the left equivalent to the equalities on the right

\begin{equation} \label{eqn:deltas}
\begin{aligned}
\delta_\V : \H_1 &\rightarrow \H_1 \otimes G^\V \qquad &(1 \otimes F) \overline{u} = u(1 \otimes F) \\
\delta_\W : \H_2 &\rightarrow \H_2 \otimes G^\W \qquad &(1 \otimes G) \overline{v} = v(1 \otimes G) \\
\delta_{\V,\W} : \H_1 &\rightarrow \H_2 \otimes G^{\V,\W} \qquad &(1 \otimes G) \overline{p} = p(1 \otimes F)
\end{aligned}
\end{equation}

For ease of notation we will start writing $uF$ for $u (1 \otimes F)$. Then we have $G^{-1}pF = \overline{p}$ and $F^{-1}qG = \overline{q}$.

Then we can check that we have a unital algebra homorphism

\begin{align*}
\H_1 &\rightarrow \H_2 \otimes (G^{\V,\W})^{op} \\
f_i &\mapsto \sum_j e_j \otimes p_{ij}^*
\end{align*}

Applying $G$ to both sides and noting that $e_j = FF^{-1}e_j$ then the map above is involutive with respect to the modified $*$-structure $\star$ on $(G^{\V,\W})^{op}$ given by $(p^*)^\star = (F^{-1}p^*F)^t$.

The universal property of $G^{\V,\W}$ implies that the homomorphism above factors as $(\id \otimes S_{\V,\W}) \delta_\V$ where $S_{\V,\W}$ is a conjugate-linear anti-morphism defned by

\begin{align*}
&S_{\V,\W} : G^{\W,\V} \rightarrow G^{\V,\W}  \qquad &q \mapsto p^* \quad &q^* \mapsto G^t \overline{p} F^{-t} \\
&S_{\W,\V} : G^{\V,\W} \rightarrow G^{\W,\V} \qquad &p \mapsto q^* \quad &p^* \mapsto F^t \overline{p} G^{-t}
\end{align*}

Since we have shown that $\C$ is a connected cogroupoid then if $G^{\V,\W}$ is non-zero, \cite{bichon1} and \cite{ours} shows that $G^{\V,\W}$ is a $G^\V - G^\W$ bigalois extension.
\end{proof}

Utilizing the work in \cite{bichon2}, we can show the existence of a bi-invariant state $\omega$ for the $G^\V - G^\W$ bigalois extension, referenced in \Cref{thrm:bichon}.

Consider $n \in \bN$ and matrices $F_i \in GL_n(\bC)$. Define $C(U_{F_1}^+, U_{F_2}^+)$ to be the unital $*$-algebra generated by the coefficients $z_{ij}$ of the $n_1\times n_2$ matrix $z = [z_{ij}] \in M_{n_1,n_2}(C(U_{F_1}^+, U_{F_2}^+))$ he relaqtions that $z$ and $F_1 \overline{z} F_2^{-1}$ are unitary where $\overline{z} = [z_{ij}^*]$. Note that if $C(U_{F_1}^+, U_{F_2}^+) \neq 0$ then $C(U_{F_1}^+, U_{F_2}^+)$ is a $C(U_{F_1}^+) - C(U_{F_2}^+)$ bigalois extension with respect to the bicomodule structure given by

\begin{align*}
\alpha_{F_1,F_2} : C(U_{F_1}^+, U_{F_2}^+) \rightarrow C(U_{F_1}^+) \otimes C(U_{F_1}^+, U_{F_2}^+) & \qquad & \alpha_{F_1,F_2}(z_{ij}) = \sum_{k=1}^{n_1} u_{ik} \otimes z_{kj} \\
\beta_{F_1,F_2} : C(U_{F_1}^+, U_{F_2}^+) \rightarrow C(U_{F_1}^+, U_{F_2}^+) \otimes C(U_{F_2}^+) & \qquad & \beta_{F_1,F_2}(z_{ij}) = \sum_{\ell=1}^{n_2} z_{i\ell} \otimes v_{\ell j}
\end{align*}

where $u = [u_{ij}]$ is the fundamental representation of $U_{F_1}^+$ and $v = [v_{ij}]$ is the fundamental representation of $U_{F_2}^+$.

\begin{theorem} \label{thrm:invariant_state} (\cite{bichon2})
Let $G$ be a compact quantum group and $(Z,\alpha)$ a left $\O(G)$-Galois extension. Let $F \in GL_n(\bC)$ be such that $G < U_F^+$ with corresponding surjective morphism $\pi : \O(U_F^+) \rightarrow O(G)$. If there exists $F_1 \in GL_{n_1}(\bC)$ and a surjective $*$-homomorphism $\sigma : \O(U_F^+, U_{F_1}^+) \rightarrow Z$ satisfying $\alpha \circ \sigma = (\pi \otimes \sigma) \alpha_{F,F_1}$ then $Z$ admits a left-invariant state $\omega : Z \rightarrow \bC$.
\end{theorem}

Thus, we have the following theorem:

\begin{theorem} \label{thrm:2}
Let $(\M_1, \V_t)$ and $(\M_2, \W_t)$ be finite quantum metric spaces. If $G^{\V,\W} \neq 0$, then there exists a faithful, bi-invariant, tracial state $\omega : G^{\V,\W} \rightarrow \bC$.
\end{theorem}

\begin{proof}
Utilizing the proof of \Cref{thrm:bigalois_extension}, we may consider the matrices $F_1 = F_\V$ and $F_2 = F_\W$. Equation \Cref{eqn:deltas} shows that we have surjective $*$-homomorphisms $\pi: \O(U_{F_\V}^+) \rightarrow G^\V$ and $\sigma : C(U_{F_Y}^+, U_{F_X}^+) \rightarrow G^{\V,\W}$ satisfying

\[ \alpha \circ \sigma = (\pi \otimes \sigma) \alpha_{F_Y,F_X}. \]

Then by \Cref{thrm:invariant_state}, $G^{\V,\W}$ admits a $G^\W-G^\V$ invariant state, and it is tracial if and only if both $G^\V$ and $G^\W$ are of Kac type.
\end{proof}

\begin{corollary} \label{thrm:monoidallyequivalent}
Let $(\M_1, \V_t)$ and $(\M_2, \W_t)$ be finite dimensional quantum metric spaces. If $G^{\V,\W}$ is non-zero, then the compact quantum groups $G^\V$ and $G^\W$ are monoidally equivalent, $G^\V \sim^{mon} G^\W$.
\end{corollary}

\begin{proof}
This is a corollary of \Cref{thrm:2}. By \Cref{thrm:bichon}, $G^\V$ and $G^\W$ are monoidally equivalent.
\end{proof}

\begin{corollary} \label{thrm:qc}
Let $(\M_1, \V_t)$ and $(\M_2, \W_t)$ be finite dimensional quantum metric spaces. Then the following are equivalent:

\begin{enumerate}
\item $(\M_1, \V_t) \cong_{A^*} (\M_2, \W_t)$
\item $(\M_1, \V_t) \cong_{C^*} (\M_2, \W_t)$
\item $(\M_1, \V_t) \cong_{qc} (\M_2, \W_t)$
\end{enumerate}
\end{corollary}

\begin{remark}
This theorem says that as soon as $\O(G^{\V,\W})$ is non-zero, then $\O(G^{\V,\W})$ admints a tracial state. This is non-trivial and it has been shown that this phenomena is not true for other games. In \cite{BGH20}, a graph homomorphism from a quantum graph to a classical graph was defined and studied. They showed that the game $*$-algebra is always non-zero when the output graph is $K_4$.
\end{remark}

By restricting our attention to classical metric spaces, we get one of the main results of the paper:

\begin{corollary} \label{thrm:qcclassical}
Let $(X,d_X)$ and $(Y,d_Y)$ be classical metric spaces. Then the following are equivalent:

\begin{enumerate}
\item $X \cong_{A^*} Y$
\item $X \cong_{C^*} Y$
\item $X \cong_{qc} Y$
\end{enumerate}
\end{corollary}

\begin{proof}
This is an immediate consequence of \Cref{thrm:qc}.
\end{proof}

\begin{remark}
An example of two classical graphs which are quantum isomorphic but not classically isomorphic is shown in \cite{atserias}. Each of the graphs have 24 vertices, and naturally give rise to classical metric spaces of size 24 which are quantum isometric but not classically isometric.
\end{remark}

%%%%%%%%%%%%%%%%%%%%%%%%%%%%%%%%%%%%%%%%%%%%%%%%%%%%%%%%%%%%

\bibliography{Nonlocal_games_and_quantum_symmetries_of_quantum_metric_spaces}{}
\bibliographystyle{plain}

\end{document}